\documentclass[11pt]{amsart}                  
\usepackage{amssymb,amsmath,amsfonts,latexsym}
\usepackage[margin=1.1in]{geometry}
\usepackage{enumerate}
\usepackage{color}
\newcommand{\complex}{{\mathbb C}}

\newcommand{\reals}{{\mathbb R}}

\newcommand{\integers}{{\mathbb Z}}

\newcommand{\calb}{{\mathcal B}}

\newcommand{\calh}{{\mathcal H}}

\newcommand{\cals}{{\mathcal S}}

\newtheorem{thm}{Theorem}[section]
\newtheorem{prop}[thm]{Proposition}

\theoremstyle{definition}

\theoremstyle{definition}

\newcommand{\be}{\begin{eqnarray}}
\newcommand{\ee}{\end{eqnarray}}

\theoremstyle{plain}
        \newtheorem{theorem}{Theorem}[section]
        
        \newtheorem{remark}[theorem]{Remark}

        \newtheorem{example}[theorem]{Example}
\begin{document}
\title{K-theory of Equivariant Quantization}
\author{Xiang Tang and Yi-Jun Yao}
\maketitle

\begin{abstract}{Using an equivariant version of Connes' Thom Isomorphism,w}e prove that equivariant $K$-theory is invariant under strict
deformation quantization for a compact
Lie group action.
\end{abstract}

\section{Introduction}\label{sec:intro}
Let $\alpha$ be a strongly continuous action of $\reals^n$ on a $C^*$-algebra $A$, and $J$ be a skew-symmetric matrix on
$\reals^n$. Rieffel \cite{rieffel:quantization} constructed a strict
deformation quantization $A_J$ of $A$ via oscillatory integrals
\begin{equation}\label{eq:star}
a\times_J b:=\int_{\reals^n\times \reals^n}
\alpha_{Ju}(a)\alpha_{v}(b)e^{2\pi i u\cdot v} du dv,
\end{equation}
for $u,v\in \reals^n$, and $a,b\in A^\infty$ (the smooth subalgebra
of $A$ for $\alpha$). Such a construction gives rise to many interesting examples
of noncommutative manifolds, e.g. quantum tori, $\theta$-deformation
of $S^4$, etc. In \cite{rieffel:k-quant}, Rieffel proved that the
$K$-theory of $A_J$ is equal to the $K$-theory of the original
algebra $A$, {by using Connes' Thom isomorphism of $K$-theory}.

In this paper, we are interested in examples that the algebra $A$ is
also equipped with a strongly continuous action $\beta$ by a compact
group $G$. When the two actions commute, {the results in \cite{rieffel:k-quant}} naturally
generalize to the equivariant setting. An easy observation is that,
as the $G$-action $\beta$ commutes with the $\reals^n$-action
$\alpha$, naturally $\alpha$ can be lifted to a strongly
continuous action $\tilde{\alpha}$ on the crossed product algebra
$A\rtimes_\beta G$. Rieffel's construction (\ref{eq:star}) applies
to the $\reals^n$-action $\tilde{\alpha}$ on $A\rtimes_\beta G$, and defines a quantization algebra $(A\rtimes_\beta
G)_J$. By the commutativity between $\alpha$ and $\beta$, we easily
check that $\beta$ lifts to a strongly continuous action
$\tilde{\beta}$ on $A_J$, and $A_J\rtimes_{\tilde{\beta}}G$ is
isomorphic to $(A\rtimes_\beta G)_J$. Now by
{the results} on the
$K$-theory of strict deformation quantization
\cite{rieffel:k-quant}, we conclude that
\[
K_\bullet(A\rtimes_\beta G)=K_\bullet((A\rtimes_\beta
G)_J)=K_\bullet(A_J\rtimes_{\tilde{\beta}}G).
\]

In this paper, we generalize the above discussion of equivariant
quantization to a situation where the actions $\alpha$ and $\beta$
do not commute. Define $GL(J)$ to be the group of invertible
matrices $g$ such that $g^t J g =J$, and $SL_n(\reals,
J):=SL_n(\reals)\bigcap GL(J)$. We remark that when $J$ is the
standard skew-symmetric matrix on $\reals^{2n}$, $GL(J)$ is the
linear symplectic group. Let $\rho: G\rightarrow SL_n(\reals, J)$ be
a group homomorphism such that
\begin{equation}\label{eq:action}
\beta_g\alpha_x=\alpha_{\rho_g(x)}\beta_g,\ \ \ \text{for any}\ g\in
G, x\in \reals^n.
\end{equation}
When $\rho$ is a trivial group homomorphism, the actions $\alpha$
and $\beta$ commute.

A natural example of such a system appears as follows.
\begin{example}\label{ex:z-2}
Let $G=\integers_2=\integers/2\integers$ act on $\reals^{2n}$ by
reflection with respect to the origin. Let $\integers^{2n}$ be the integer lattice in
$\reals^{2n}$. The $2n$-torus ${\mathbb
T}^{2n}=\reals^{2n}/\integers^{2n}$ inherits an action of
$\integers_2$ from the $\integers_2$ action on $\reals^{2n}$. The
group $\reals^{2n}$ acts on $\reals^{2n}$ by translation and
descends to act on $\mathbb {T}^{2n}$. Let $A$ be the $C^*$-algebra
of continuous functions on ${\mathbb T}^{2n}$, and $J$ be the
standard symplectic matrix on $\reals^{2n}$. The action $\alpha$
(and $\beta$) of $\reals^{2n}$ (and $\integers_2$) on $A$ is the
dual action of the corresponding actions on ${\mathbb T}^{2n}$. We
easily check that Eq. (\ref{eq:action}) holds in this case with
$\rho$ being the natural inclusion $\integers_2\hookrightarrow
SL_{2n}(\reals, J)$.
\end{example}
Different from the case where the actions $\alpha$ and $\beta$
commute, for a nontrivial $\rho:G\rightarrow SL_n(\reals, J)$, the
$\reals^n$-action $\alpha$ on $A$ does not lift naturally to an
action on $A\rtimes_\beta G$. Therefore, we cannot apply Rieffel's
deformation construction to the algebra $A\rtimes_\beta G$.
Nevertheless, a simple calculation shows that
\[
\beta_g(a\times_J b)=\beta_g(a)\times_J \beta_g(b),\ \
\beta_g(a^*)=\beta_g(a)^*,
\]
which shows that the $G$-action $\beta$ is still well-defined on
$A_J$. Accordingly, we can consider the crossed product algebra
$A_J\rtimes_\beta G$. Applying this construction to Ex.
\ref{ex:z-2}, we obtain $A_J\rtimes_\beta \integers_2$, which is
well studied in literature, e.g. \cite{ech-luck-phil:k-qtorus},
\cite{farsi:qorbfld}, \cite{ku:k-theory}, and \cite{walters:k-theory}.

In this paper, we prove the following theorem about the $K$-theory
groups of $A_J\rtimes_\beta G$.
\begin{theorem}\label{thm:main-k}
If the actions $\alpha$,  $\beta$ and the group homomorphism $\rho$
satisfy (\ref{eq:action}), then
\[
K_\bullet(A_J\rtimes_\beta G)\cong K_{\bullet}(A\rtimes_\beta G),\ \
\ \bullet=0,1.
\]
\end{theorem}

The proof of this theorem will be presented in the next section.
As applications of our theorem, we recover some results of
\cite{ech-luck-phil:k-qtorus} on the computation of the
$K$-groups of ${\mathbb Z}_i$-quantum tori for $i=2,3,4,6$,  and we 
{also apply}
 these results to the $\theta$-deformation \cite{connes-landi:theta}
of $S^4$.\\

\noindent{\bf Acknowledgments:} We would like to thank Professors S. Echterhoff and N. Higson for explaining the relationship between the equivariant
Thom isomorphism theorem (Theorem \ref{thm:equ-thom}) and the Connes-Kasparov conjecture. We also want to thank Professor H. Li for interesting
discussions and comments which greatly helped us to improve the readability of the paper. {We thank Professor H. Oyono-Oyono for helping us
to remove a separability assumption in a previous version.} And we are grateful to an anonymous referee for pointing out
one mistake and various places to improve accuracy in a previous version of this paper. Tang's research is partially
supported by NSF grant 0900985. Tang would like to thank the School
of Mathematical Sciences of Fudan University and the Max-Planck
Institute for their warm hospitality of his visits. Yao's research
is partially supported by NSF grant 0903985 and NSFC grant
10901039
{and 11231002}.

\section{Proof of the main theorem}
Our proof of Theorem \ref{thm:main-k} is an equivariant
generalization of Rieffel's
proof in \cite{rieffel:k-quant}.  We first prove the theorem under the assumption that $A$ is separable. Following
\cite{rieffel:k-quant}, we will decompose our
proof into 3 steps.

\bigskip

\noindent{\bf Step I.} Following 
{the notations in}
\cite{rieffel:k-quant}, we let $\calb^A$ be the space of smooth
$A$-valued functions on $\reals^n$ whose derivatives together with
themselves are bounded on $\reals^n$. Let $\cals^A$ be the space of
$A$-valued Schwartz functions on $\reals^n$. The  integral
\[
\langle f, g\rangle_A:=\int f(x)^*g(x)dx
\]
defines an $A$-valued inner product on $\cals^A$. Rieffel generalized the definition to
$\calb^A$ by using oscillatory integrals. Namely, given $J$, we
define a product on $\calb^A$ by
\[
(F\times_J G)(x):=\int F(x+Ju)G(x+v)e^{2\pi i u\cdot v}dudv,\qquad F,G\in
\calb^A.
\]
Furthermore, $\calb^A$ acts on $\cals^A$ by
\[
(L^J_F f)(x):=\int F(x+Ju)f(x+v)e^{2\pi i u\cdot v}dudv,\qquad F\in
\calb^A,\ f\in \cals^A.
\]
The above two integrals are both oscillatory ones.
Via the $A$-valued inner product on $\cals^A$, we can equip
$\calb^A$ with the operator norm $\|\ \|_J$,  and obtain a pre-$C^*$-algebra
$(\calb^A_J,\times_J, \|\ \|_J)$. Denote the corresponding
$C^*$-algebra by $\overline{\calb}^A_J$.
Meanwhile, $\cals^A$ viewed as a $*$-ideal of $\calb^A_J$  (cf.
Rieffel, \cite{rieffel:quantization}), denoted by $\cals^A_J$, can be completed
into $\overline{\cals}^A_J$.

With an action $\alpha$ of $\reals^n$ on $A$, Rieffel \cite[Prop.
1.1]{rieffel:k-quant} introduced a strongly continuous
$\reals^n$-action $\nu$ on $\overline{\calb}^A_J$ and also on
$\overline{\cals}^A_J$ by
\[
(\nu_t(F))(x):=\alpha_t(F(x-t)).
\]
The fixed point subalgebra of this action $\nu$ is identified
\cite[Prop. 2.14]{rieffel:k-quant} with the $C^*$-subalgebra of
$\overline{\calb}^A_J$ generated by elements
\[
\tilde{a}(x):=\alpha_x(a),\qquad a\in A^\infty,
\]
which is exactly $A_J$.

{In \cite[Thm. 3.2]{rieffel:k-quant}, it is }proved that $A_J$ is
strongly Morita equivalent to $\overline{\cals}^A_J\rtimes_\nu
\reals^n$. We will generalize this theorem to the equivariant
setting with the $G$-action $\beta$. We introduce the $G$-action $\overline{\beta}$ on ${\calb}^A_J$ by
\[
\overline{\beta}_g(F)(x):=\beta_g(F(\rho_{g^{-1}}(x))).
\]
The exactly same arguments as {in} \cite[Prop. 1.1]{rieffel:k-quant} prove
that the $G$-action $\overline{\beta}$ is strongly continuous on
$\cals^A$, therefore so is it on $\overline{\cals}^A_J$.

\begin{prop}\label{prop:fixed-pt}The crossed product algebras $A_J\rtimes_\beta
G$ and $(\overline{\cals}^A_J\rtimes_\nu
\reals^n)\rtimes_{\bar{\beta}}G$ are strongly Morita equivalent.
\end{prop}
\begin{proof}We will apply Combes' theorem \cite[Sec. 6]{combes:crossed}
on equivariant Morita equivalence 
 {after proving} that the $G$-actions $\beta$ and $\bar{\beta}$ are Morita equivalent, which
 {will imply} the Morita equivalence we seek.

According to \cite{combes:crossed}, two $G$-actions
$\beta^1,\beta^2$ on $A$ and $B$ are Morita equivalent if there is a
strong Morita equivalence bimodule $X$ between $A$ and $B$ such that
there is a $G$-action $\beta$ on $X$ satisfying
\[
\begin{split}
\beta_g(a\xi)=\beta^1_g(a)\beta_g(\xi),\ \ \qquad\qquad
\qquad&\beta_g(\xi
b)=\beta_g(\xi)\beta^2_g(b),\\
 _A\langle \beta_g(\xi_1),
\beta_g(\xi_2) \rangle=\beta^1_g(_A \langle \xi_1,
\xi_2\rangle),\qquad &\langle \beta_g(\xi_1),
\beta_g(\xi_2)\rangle_B=\beta^2_g(\langle \xi_1, \xi_2\rangle_B).
\end{split}
\]
for $\xi, \xi_1, \xi_2\in X$.
\smallskip

Rieffel \cite{rieffel:k-quant} constructed a Morita equivalence
bimodule between $A_J$ and $\overline{\cals}^A_J\rtimes _\nu \reals^n$. We
recall it now. Let $C_\infty(\reals^n, A)$ be the $C^*$-algebra of
$A$-valued functions on $\reals^n$ that vanish at infinity. Let
$\tau$ be the $\reals^n$-action on $\calb^A_J$ by translation,
$(\tau_tF)(x)=F(t+x)$, and $\mu$ be the action of $\reals^n$ on
$C_\infty(\reals^n,A)$ by
\[
\mu_s(f)(x)=e^{2\pi i s\cdot x}f(x).
\]
Define an action $\alpha$ of $\reals^n$ on $\overline{\cals}^A_J$ by
\[
\alpha_t(F)(x)=\alpha_t(F(x)).
\]
Both $\mu$ and $\tau$ act on $C_\infty(\reals^n,A)$ and  their combination gives an
action of the Heisenberg group $H$ of
dimension $2n+1$ on $C_\infty(\reals^n, A)$. This
Heisenberg group action commutes with $\alpha$ and
defines an $H\times \reals^n$-action $\sigma$ on $C_\infty(\reals^n,
A)$. Define $X_0$ to be the subspace of $C_\infty(\reals^n, A)$ of
$\sigma$-smooth vectors. Rieffel \cite[Prop. 2.2]{rieffel:k-quant} proved that $X_0$ is a $*$-subalgebra of $\cals^A_J$ for any $J$, and a suitable
completion $\overline{X}_0$ of $X_0$ serves as a strong Morita equivalence bimodule, which we refer to \cite{rieffel:k-quant} for details.

Define a right $A_J$-module structure on $\overline{X}_0$ by
identifying $A_J$ with the subspace of $\nu$-invariant vectors in
$\overline{\calb}^A_J$, i.e.
\[
f\cdot a:=f\times_J \tilde{a},\qquad \text{for}\ a\in A^\infty
\]
where $\tilde{a}\in C^\infty(\reals^n, A)$ is defined by $\tilde{a}(x)=\alpha_x(a)$.
The algebra
$\overline{\cals}^A_J\rtimes_{\nu}\reals^n$ acts on $\overline{X}_0$ by
\[
\psi(f):=\int \psi(t)\times_J\nu_{t}(f)dt,\qquad \psi\in
\overline{\cals}^A_J\rtimes_{\nu}\reals^n,\ f\in \overline{X}_0.
\]
We define an $\overline{\cals}^A_J\rtimes_{\nu}\reals^n$-valued inner
product on $\overline{X}_0$ by
\[
_{\overline{\cals}^A_J\rtimes_{\nu}\reals^n}\langle
f,g\rangle(x):=f\times_J \nu_x(g^*),\qquad x\in \reals^n,\ f,g\in
\overline{X}_0,
\]
and an $A_J$-valued inner product on $\overline{X}_0$ by
\[
\langle f,g\rangle_{A_J}:=\left(\int \alpha_t(f^*\times_J g(-t))dt\right),\qquad f,g\in
\overline{X}_0.
\]
{We also know from \cite{rieffel:k-quant} that}
$\left(\overline{X}_0,\,\,
_{\overline{\cals}^A_J\rtimes_{\nu}\reals^n}\langle\ ,\ \rangle,
\langle\ ,\ \rangle_{A_J}\right)$ is a strong Morita equivalence
bimodule between $\overline{\cals}^A_J\rtimes_\nu \reals^n$ and
$A_J$.

We easily check the following identities between the actions
\[
\overline{\beta}_g\alpha_t=\alpha_{\rho_g(t)}\overline{\beta}_g,\qquad
\overline{\beta}_g\tau_t=\tau_{\rho_g(t)}\overline{\beta}_g, \qquad
\overline{\beta}_g\mu_{t}=\mu_{(\rho_g^T)^{-1}(t)}\overline{\beta}_g,\qquad
g\in G,\ t\in \reals^n.
\]
where $\rho_g^{T}$ is the transpose of $\rho_g$. These identities
show that the $G$-action $\overline{\beta}$ on $C_\infty(\reals^n,
A)$ preserves the subspace $X_0$ of $\sigma$-smooth vectors. Using
the property that $\beta$ and $\overline{\beta}$ act strongly
continuously on $\overline{\cals}^A_J$, we can easily check that
$\beta$ and $\overline{\beta}$ are Morita equivalent
$G$-actions in the sense of Combes \cite{combes:crossed}.
Therefore, $A_J\rtimes_\beta G$ is strongly Morita equivalent to
$(\overline{\cals}^A_J\rtimes_{\nu}\reals^n)\rtimes_{\overline{\beta}}
G$.
\end{proof}

As $A$ is separable, $A$ has a countable approximate identity. This implies that \cite[Cor. 3.3]{rieffel:k-quant} $A_J$ (and $\overline{\cals}^A_J$) has
a countable approximate identity. Accordingly, $A_J\rtimes_\beta G$ (and $(\overline{\cals}^A_J\rtimes_{\nu}\reals^n)\rtimes_{\overline{\beta}}
G$) also has a countable approximate identity, and therefore has strictly positive elements. This together with the above Morita equivalence result
shows that $A_J\rtimes_\beta G$ and $(\overline{\cals}^A_J\rtimes_{\nu}\reals^n)\rtimes_{\overline{\beta}}
G$ are stably isomorphic. As stably isomorphic $C^*$-algebras have isomorphic $K$-groups,
we conclude that
\[
K_\bullet(A_J\rtimes_\beta G)\cong
K_{\bullet}((\overline{\cals}^A_J\rtimes_{\nu}\reals^n)\rtimes_{\beta}G).
\]
\noindent{\bf Step II.} {As we know, one powerful tool in dealing with the $K$-theory of $C^*$-algebras is Connes' Thom isomorphism,
which remains to this day one of the few ways to prove isomorphism results of $K$-groups for crossed products.} Let $\complex_n$ be the complex Clifford
algebra associated with $\reals^n$. We first observe that the semidirect product group $\reals^n\rtimes_{\rho} G $  is amenable, hence by
Kasparov\footnote{In \cite[\S 6, Thm. 2.]{kasp:nov}, the connectivity of the group $G$  is assumed. But this assumption can be easily dropped using the
same idea of the proof.} \cite[\S 6, Thm. 2.]{kasp:nov} we know that for a separable $G$-$C^*$-algebra $B$, there exists an isomorphism from
$KK^i(\complex, B\rtimes (\reals^n\rtimes G))$ to $KK^i(\complex, ((B\otimes \complex_n)\rtimes G)$.
{In other words, we need to use the following equivariant Thom isomorphism Theorem, which is a generalization of Connes' Thom isomorphism
Theorem \cite{connes:thom}. This is a key ingredient of the whole approach.}
\begin{theorem}\label{thm:equ-thom}Let $\reals^n$ and $G$ act
strongly continuously on a separable $C^*$-algebra $B$ with the
actions denoted by $\alpha$ and $\beta$. Let $\rho:G\rightarrow
GL(n, \reals)$. If the actions $\alpha$ and $\beta$ satisfy Equation
(\ref{eq:action}), then
\[
K_\bullet((B\rtimes_\alpha
\reals^n)\rtimes_{\beta} G)\cong
K^G_\bullet\big(B\rtimes_\alpha
\reals^n\big)\cong K^G_{\bullet}(B\otimes \complex_n)\cong
K_{\bullet}((B\otimes \complex_n)\rtimes_{\beta} G),
\]
where $\complex_n$ is the complex Clifford algebra associated with
$\reals^n$.
\end{theorem}
Taking $B=\overline{\cals}^A_J$ in the above theorem which is separable (as $A$ is separable), we conclude that
$K_\bullet((\overline{\cals}^A_J\otimes \complex_n)\rtimes_{\bar{\beta}} G)$
is isomorphic to
$K_{\bullet}((\overline{\cals}^A_J\rtimes_{\nu}\reals^n)\rtimes_{\overline{\beta}}G)$.
\medskip

\noindent{\bf Step III.} Rieffel proved \cite[Prop.
5.2]{rieffel:quantization} that there is an isomorphism
\begin{equation}\label{eq:iso}
\overline{\cals}^A_J\cong A\otimes {\mathcal K}\otimes C_\infty(V_0),
\end{equation}
where ${\mathcal K}$ is the algebra of compact operators on an infinite dimensional separable Hilbert
space $\calh$, and $V_0$ is the kernel of $J$ in $\reals^n$. Let $U$ be the
orthogonal complement of $V_0$ in $\reals^n$. It is easy to check that $U$ is a $J$-invariant subspace, and both $U$ and $V_0$ are
$G$-invariant subspaces. As $G$ is compact, there is a $G$-invariant
complex structure on $U$ compatible with $J|_U$ (viewed a symplectic form on $U$). Without loss of
generality, we will just assume that $G$ preserves the standard
Euclidean structure on $U$. The key observation in the proof of
\cite[Prop. 5.2]{rieffel:quantization} is that when $A$ is the
trivial $C^*$-algebra $\complex$ and $J$ invertible, $\overline{\cals}^\complex_J$ is
naturally identified as the space of compact operators, still denoted by $\mathcal K$, on the
subspace $\calh$ of $L^2(U)$ generated by elements
\[
g(\bar{z})e^{-\frac{\|z\|^2}{2}},
\]
where $g$ is an anti-holomorphic function. As $\calh$ is a
$G$-invariant subspace, we can conclude that Rieffel's isomorphism
(\ref{eq:iso}) is $G$-equivariant (note that $G$ acts on ${\mathcal {K}}$ by conjugation). By Combes' result on G-equivariant Morita equivalence,
$(A\otimes {\mathcal K}\otimes
C_\infty(V_0)\otimes\complex_n)\rtimes_{\bar{\beta}} G$ is strongly
Morita equivalent to $(A\otimes C_{\infty}(V_0)\otimes
\complex_n)\rtimes_{\bar{\beta}} G$.

Now we look at the decomposition
of $\reals^n$ as $V_0\oplus U$. The Clifford algebra $\complex_n$
associated with $\reals^n$ is $G$-equivariantly isomorphic to
$\complex_{V_0}\otimes \complex_{U}$, where $\complex_{V_0}$ and
$\complex_U$ are the complex Clifford algebras associated with $V_0$ and
$U$, respectively. Notice that $J$ restricts to define a symplectic form on $U$,
and that the action of $G$ preserves both the restricted $J$ and the
metric on $U$. Therefore the $G$-action on $U$ is $spin^c$. Hence,
the algebra $(A\otimes C_{\infty}(V_0)\otimes
\complex_n)\rtimes_{\bar{\beta}} G$ is $KK$-equivalent to $(A\otimes
C_{\infty}(V_0)\otimes \complex_{V_0})\rtimes_{\bar{\beta}} G$.
Again by the $G$-equivariant Thom isomorphism Thm. \ref{thm:equ-thom} for the trivial $V_0$ action on $A$, we conclude that
\[
\begin{split}
K_{\bullet}((\overline{\cals}^A_J\otimes\complex_n)\rtimes_{\bar{\beta}}
G)&=K_{\bullet}((A\otimes {\mathcal K}\otimes
C_\infty(V_0)\otimes\complex_n)\rtimes_{\bar{\beta}}
G)\\
&=K_\bullet(\big(A\otimes C_\infty(V_0)\otimes
\complex_{V_0}\big)\rtimes_{\bar{\beta}} G)=K_\bullet(A\rtimes_\beta
G).
\end{split}
\]

Summarizing Step I-III, we have the following equality,
\[
K_\bullet(A_J\rtimes_\beta G)\stackrel{\text{Step
I}}{===}K_\bullet((\overline{\cals}^A_J\rtimes_{\nu}
\reals^n)\rtimes_{\bar{\beta}} G)\stackrel{\text{Step
II}}{===}K_{\bullet}((\overline{\cals}^A_J\otimes\complex_n)\rtimes_{\bar{\beta}}
G)\stackrel{\text{Step III}}{===}K_\bullet(A\rtimes_\beta G).
\]
{This completes the proof of Theorem \ref{thm:main-k} under the assumption that $A$ is separable. For a general $C^*$-algebra $A$, we can write
$A$ as an inductive limit of a net $A^I$ of separable $\reals^n\rtimes_\rho G$-algebras. Then $A_J$ is an inductive limit of the net $A^I_J$
of separable $G$-algebras. As $K$-groups commutes with inductive limit, we conclude that
\[K_\bullet(A_J\rtimes_\beta G)=\lim\limits_I K_\bullet(A_J^I\rtimes_\beta G)=\lim\limits_I K_\bullet(A^I\rtimes_\beta G)=K_\bullet(A\rtimes_\beta G).
\]
This completes the proof of Theorem \ref{thm:main-k} for general $C^*$-algebras.}

\section{Examples}
In this section, we discuss some applications of Theorem
\ref{thm:main-k}.
\subsection{Noncommutative toroidal orbifolds}
We identify a 2-torus ${\mathbb T}^2$ by $\reals^2/\integers^2$. $\reals^2$
acts on itself by translation and induces an action $\alpha$ on
${\mathbb T}^2$. For $\theta\in\reals$, we consider the symplectic form $J=\theta dx_1\wedge dx_2$ on
$\reals^2$. The group $SL_2(\integers)$ acts on $\reals^2$
preserving the lattice $\integers^2$ and therefore also acts on
${\mathbb T}^2$, which is denoted by $\beta$. Inside $SL_2(\integers)$, there
are cyclic subgroups generated by
\[
\begin{split}
\sigma_2=\left(\begin{array}{cc}-1&0\\
0&-1\end{array}\right),\qquad\qquad &\sigma_3=\left(\begin{array}{cc}-1&-1\\
1&0\end{array}\right)\\
\sigma_4=\left(\begin{array}{cc}0&-1\\
1&0\end{array}\right),\qquad\qquad\ \ &\sigma_6=\left(\begin{array}{cc}0&-1\\
1&1\end{array}\right).
\end{split}
\]
The element $\sigma_i$ generates a cyclic subgroup $\integers_i$ of
$SL_2(\integers)$ of order $i=2,3,4,6$. In this example, the group
$SL_2(\reals, J)$ is identical to the group $SL_2(\reals)$. Define
$\rho:\integers_i\to SL_2(\reals)$ to be the inclusion. And it is
straightforward to check the actions $\beta$ of $\integers_i$ on
${\mathbb T}^2$, $\rho$ of $\integers_i$ on $\reals^2$, and $\alpha$ of
$\reals^2$ on ${\mathbb T}^2$ satisfy Eq. (\ref{eq:action}). As is explained
in Sec. \ref{sec:intro}, the group $\integers_i$ naturally acts on
Rieffel's deformation $A_J$, which is the quantum torus $A_\theta$.
Theorem \ref{thm:main-k} states that
\[
K_\bullet(A_J\rtimes \integers_i)=K_\bullet(A\rtimes \integers_i).
\]
We recover with a completely different proof the result of \cite[Cor.
2.2]{ech-luck-phil:k-qtorus}. We have brought the question of computation of
 $K$-groups of these noncommutative orbifolds to a purely topological setting, and we refer
 to \cite{ech-luck-phil:k-qtorus}  and references therein for the explicit computation of the $K$-groups
 of the undeformed algebras $A\rtimes \integers_i$, $i=2,3,4,6$. For example, when $i=2$, the $K$-groups of $A\rtimes\integers_2$ are
\[
K_\bullet(A\rtimes \integers_2)\cong\left\{\begin{array}{ll}\integers^6,&\bullet=0,\\ 0,&\bullet=1.\end{array}\right.
\]
\subsection{Theta deformation}
Consider a 4-sphere $S^4$ centered at $(0,0,0,0,0)$ in $\reals^5$
with radius $1$. In coordinates, it is the set
\[
\left\{(x_1, \cdots, x_5)|
x_1^2+x_2^2+x_3^2+x_4^2+x_5^2={1}\right\}.
\]
Defines ${\mathbb T}^2$-action on $S^4$ by, for $0\leq t_1, t_2 < 2\pi$,
\[
\big((t_1,t_2),(x_1, \cdots, x_5)\big)\longrightarrow
(x_1, \cdots, x_5)\left(\begin{array}{ccccc}\cos(t_1)&\sin(t_1)&0&0&0\\
-\sin(t_1)&\cos(t_1)&0&0&0\\0&0&\cos(t_2)&\sin(t_2)&0\\
0&0&-\sin(t_2)&\cos(t_2)&0\\0&0&0&0&1
\end{array}\right).
\]
The same formula as above also defines an $\reals^2$-action $\alpha$
on $S^4$. The action $\beta$ of $\integers_2$  on $S^4$ is
by reflection
\[
(\sigma_2, (x_1, \cdots, x_5))\longrightarrow (x_1, -x_2, x_3, -x_4,
x_5).
\]
The group $\integers_2$ also acts on $\reals^2$ by reflection
\[
\rho: \sigma_2\longrightarrow \left(\begin{array}{cc}-1&0\\
0&-1
\end{array}\right).
\]
On $\reals^2$, for $\theta\in \reals$, consider the same symplectic form $J=\theta
dx_1\wedge dx_2$. It is easy to check that the actions $\alpha,
\beta, \rho$ satisfy Eq. (\ref{eq:action}). Consider the algebra
$C(S^4)$ of continuous functions on $S^4$. Rieffel's construction
defines a deformation $C(S^4_\theta)$ of $C(S^4)$ by $J$ and the action
$\alpha$, which is the $\theta$-deformation \cite{connes-landi:theta}
introduced by Connes and Landi. As is explained in Sec.
\ref{sec:intro}, $\integers_2$ acts strongly continuously on
$C(S^4_\theta)$. Theorem \ref{thm:main-k} states that
\[
K_\bullet(C(S^4)\rtimes \integers_2)=K_\bullet(C(S^4_\theta)\rtimes
\integers_2).
\]
The $K$-theory of $C(S^4)\rtimes \integers_2$ can be computed
\cite{phi:k-theory} topologically as the Grothendieck group of the monoid of all isomorphism classes of
$\integers_2$-equivariant vector bundles on $S^4$.

Notice that the quotient $S^4/\integers_2$ is an orbifold
homeomorphic to $S^4$. As an orbifold, $S^4/\integers_2$
\cite{mo-pr:covering} has a good covering $\{U_i\}$ such that each
$U_i$ and any none empty finite intersection $U_{i_1}\cap \cdots\cap
U_{i_k}$ is a quotient of a finite group action on $\reals^4$. Such
a good covering allows to compute the topological
$\integers_2$-equivariant $K$-theory of $S^4$ by the \v{C}ech
cohomology on $S^4/\integers_2$ of the sheaf ${\mathcal
K}^\bullet_{\integers_2}$ introduced by Segal
\cite{segal:equivariant}. The restriction of ${\mathcal
K}^\bullet_{\integers_2}$ to an open chart $U$ of $S^4/\integers_2$
is defined to be the $\integers_2$-equivariant $K$-theory of
$\pi^{-1}(U)$ with $\pi$ the canonical projection $S^4\to
S^4/\integers_2$. Locally, when $U$ is sufficiently small, we can
compute $K^\bullet_{\integers_2}(U)$ to be
$K^\bullet(\pi^{-1}(U)^{\sigma_2})\oplus K^\bullet(U)$, where
$\pi^{-1}(U)^{\sigma_2}$ is the $\sigma_2$-fixed point submanifold.
When $\bullet=0$, it is equal to
$\integers|_{\pi^{-1}(U)^{\sigma_2}}\oplus \integers_{U}$, and when
$\bullet=1$, it is zero. Gluing this local computation by the
Mayer-Vietoris sequence, we conclude that
\[
K_0(C(S^4_\theta)\rtimes \integers_2)=\integers^4,\qquad
K_1(C(S^4_\theta)\rtimes \integers_2)=0.
\]

\begin{remark}
We observe that in the above example, the group $\integers_2$ is not
essential. Our computations generalize to
$K_\bullet(C^\infty(S^4_\theta)\rtimes \integers_i)$, for $i=3,4,6$.
\end{remark}

\vspace{2mm}

{\small \noindent{Xiang Tang}, Department of Mathematics, Washington
University, St. Louis, MO, 63130, U.S.A.,
Email: xtang@math.wustl.edu.

\vspace{2mm}

\noindent{Yi-Jun Yao}, School of Mathematical Sciences, Fudan
University, Shanghai 200433, P.R.China., Email:
yaoyijun@fudan.edu.cn.

}
\end{document}